\documentclass[11pt,reqno]{amsart}

\usepackage{amsmath,amssymb,amsthm}
\usepackage[margin=1.1in]{geometry}
\usepackage{microtype}
\usepackage{tikz}
\usepackage[colorlinks=true,linkcolor=blue,citecolor=blue]{hyperref}

\theoremstyle{plain}
\newtheorem{theorem}{Theorem}[section]
\newtheorem{proposition}[theorem]{Proposition}
\newtheorem{lemma}[theorem]{Lemma}
\newtheorem{corollary}[theorem]{Corollary}
\theoremstyle{definition}
\newtheorem{remark}[theorem]{Remark}
\newtheorem{example}[theorem]{Example}

\newcommand{\R}{\mathbb{R}}
\newcommand{\Z}{\mathbb{Z}}
\newcommand{\C}{\mathbb{C}}
\newcommand{\HH}{\mathbb{H}}
\newcommand{\SU}{\mathrm{SU}(2)}
\newcommand{\SO}{\mathrm{SO}(3)}
\newcommand{\Pin}{\mathrm{Pin}(2)}
\newcommand{\tr}{\operatorname{tr}}
\newcommand{\Inat}{I^{\natural}}
\newcommand{\gr}{\operatorname{gr}}
\newcommand{\Sig}{\Sigma}
\newcommand{\pill}{\mathcal{P}}
\newcommand{\bp}{\mathfrak{b}}
\newcommand{\td}{\widetilde d}

\begin{document}

\title[Traceless characters and instanton gradings for $2$-bridge and $(3,n)$ knots]
{Traceless $\SU$ characters and $\Z/4$ instanton gradings\\ for two-bridge and $(3,n)$-torus knots}

\author{Bernd J. Wuebben}
\date{First version: July 27, 2026. This version: October 6, 2026}

\begin{abstract}
For every torus knot, the rank of the differential of the Daemi--Scaduto S-complex of singular instanton
homology is determined by the signature and the Alexander polynomial; this follows from theorems of Li--Ye
and Daemi--Scaduto. For the torus knots $T(3,n)$ the differential has rank zero or one according to $n$
modulo $6$, and its rank equals Daemi and Scaduto's Frøyshov-type invariant; they proved this in two of the
four residue classes and for $n=4,5$, and Hedden, Herald and Kirk computed it for $n\le38$. Following
Lewallen, Klassen and Hedden--Herald--Kirk, we describe the generators of the instanton complexes and their
$\Z/4$ gradings for two-bridge knots, whose irreducible traceless characters are all binary dihedral, and
for $T(3,n)$, where at most one is. We also give a proof through Khovanov homology, using Sch\"utz's
computation for three-braids, of Li and Ye's computation of the rank of the instanton homology of $T(3,n)$,
and compare the differential of the knot $8_{19}$ with the single pillowcase bigon of Hedden, Herald and
Kirk.
\end{abstract}

\maketitle

\section{Introduction}

\subsection{The Atiyah--Floer conjecture and its knot version}\label{sec:intro-af}
Let $Y$ be a closed oriented $3$-manifold with a Heegaard splitting $Y=Y_0\cup_\Sigma Y_1$. Two Floer
theories are associated with this picture. On the gauge-theoretic side is Floer's instanton homology
$HF^{\mathrm{inst}}(Y)$, generated by flat $\SU$ connections and with a differential counting
anti-self-dual instantons on $Y\times\R$. On the symplectic side, the moduli space $M(\Sigma)$ of flat
$\SU$ connections on $\Sigma$ is symplectic, the handlebodies determine Lagrangians $L_0,L_1\subset
M(\Sigma)$, and one forms the Lagrangian Floer homology $HF^{\mathrm{symp}}(L_0,L_1)$. The Atiyah--Floer
conjecture~\cite{Atiyah} asserts that these are isomorphic. Dostoglou and Salamon proved the analogous
statement for mapping tori~\cite{DS}. For Heegaard splittings, Salamon and Wehrheim developed instanton
Floer theory with Lagrangian boundary conditions~\cite{SW}, and Daemi, Fukaya and Lipyanskiy establish a
version of the conjecture for admissible $\SO$-bundles~\cite{DFL}; see Auroux's survey~\cite{Auroux} for
the place of the conjecture in Lagrangian Floer theory. The original statement for the trivial bundle
over an integral homology sphere remains open.

There is a knot-theoretic analogue in which the pillowcase plays the role of $M(\Sigma)$. Let
$K\subset S^3$ be a knot, and cut $S^3$ along a Conway sphere meeting $K$ in four points into two
$2$-string tangles. A representation $\rho\colon\pi_1(S^3\setminus K)\to\SU$ is \emph{traceless} if every
meridian maps to a traceless element (conjugate to $\operatorname{diag}(i,-i)$); this is the boundary
condition of Kronheimer--Mrowka's singular instanton theory~\cite{KM1,KM2}. Restricting traceless flat
connections of each tangle to the four-punctured sphere and applying holonomy perturbations, one obtains
immersed curves $L_0,L_1$ in the \emph{pillowcase}
\[
\pill=(\R/2\pi\Z)^2/\{(\gamma,\theta)\sim(-\gamma,-\theta)\},
\]
the traceless $\SU$ character variety of the four-punctured sphere: a two-sphere with four $\Z/2$
orbifold points~\cite{Lin,HHK1}. Hedden, Herald and Kirk conjectured that for suitable decompositions, in
which one tangle is trivial and carries the earring (the small linking circle that Kronheimer and Mrowka
add to the knot to define the reduced theory; Section~\ref{sec:2bridge-pillow}), the Lagrangian
Floer homology of $(L_0,L_1)$ is isomorphic to Kronheimer--Mrowka's reduced singular instanton homology
$\Inat(K)$ as a relatively $\Z/4$-graded group~\cite[Introduction and Conj.~6.5]{HHK2}. Cazassus, Herald,
Kirk and Kotelskiy proposed that an isomorphism with $\Inat(K)$ holds for arbitrary two-tangle
decompositions once $L_0$ and $L_1$ carry suitable bounding cochains~\cite[Conj.~D]{CHKK}. (A bounding
cochain on an immersed curve $L$ is an element $b$ of the self-Floer complex $CF(L,L)$ that satisfies the
Maurer--Cartan equation $\sum_k\mathfrak m_k(b,\dots,b)=0$; it deforms the Floer differential by counts of higher
polygons~\cite{Auroux,CHKK}. Cazassus--Herald--Kirk--Kotelskiy propose figure-eight bubbling, a
degeneration of quilted strips that can occur when Lagrangian correspondences are composed, as the
mechanism that produces such cochains~\cite[\S10.4]{CHKK}.)

\subsection{What this note contains}
This note collects, for two families of knots, the data on which the pillowcase approach rests: the
generators (the traceless characters), their $\Z/4$ gradings, and the differential of the Daemi--Scaduto
S-complex. Most of the individual statements are known, and each is attributed where it is stated. What
the note adds is a uniform treatment of the family $T(3,n)$ in all four residue classes of $n$ modulo $6$
(Sections~\ref{sec:3n} and~\ref{sec:grading}); the rank formula of Corollary~\ref{cor:torus}, which
follows at once from work of Li--Ye and Daemi--Scaduto; and a proof through Khovanov homology, using
Sch\"utz's computation for three-braids, that $\operatorname{rank}\Inat(T(3,n))=\|\Delta_{T(3,n)}\|_1$
(Lemma~\ref{lem:homrank}). In particular, for $n\equiv4,5\pmod6$ and $n\ge10$ the rank-one statement
of Theorem~\ref{thm:C}(ii) is deduced here from Corollary~\ref{cor:torus} and the recursion for $h$
in~\cite[Prop.~6.11]{DScaduto}.

Traceless character varieties of two-bridge knots also enter the work of
Dai--Mallick--Taniguchi, who use them to compute, in part, symmetry actions on singular instanton
complexes~\cite{DMT}.

\subsection{Results}\label{sec:intro-results}
Section~\ref{sec:pillow} recalls the pillowcase and how traceless characters give generators.

\begin{theorem}[Lewallen; Hedden--Herald--Kirk; Nagasato--Yamaguchi]\label{thm:A}
Let $K=\bp(p,q)$ be the two-bridge knot of fraction $p/q$, normalized so that $q$ is odd; then
$p=\det K$. Every irreducible
traceless $\SU$ representation of $\pi_1(S^3\setminus K)$ is binary dihedral. There are exactly
$(p-1)/2$ of them up to conjugacy: writing $\omega$ for the angle between the images of the two meridian
generators, they occur at $\cos\omega=\cos(2\pi k/p)$ for $k=1,\dots,(p-1)/2$, independently of $q$.
Equivalently, the traceless Riley polynomial (Riley's polynomial at meridian eigenvalue $s=i$, made monic;
Section~\ref{sec:2bridge}) in $u=2\cos\omega-2$ is
\[
\phi_p(u)=\prod_{k=1}^{(p-1)/2}\Bigl(u+4\sin^2\tfrac{\pi k}{p}\Bigr),
\]
of degree $(p-1)/2$, with integer coefficients, $\phi_p(0)=p=\det K$, and sum of roots $-p$.
\end{theorem}

The binary-dihedral statement and the count are due to Lewallen~\cite[Prop.~2.3, Thm.~2.4 and the
remark after Thm.~2.5]{Lewallen}; Hedden--Herald--Kirk give another argument~\cite[Thm.~10.2]{HHK1},
and Nagasato--Yamaguchi prove the statement for trace-free $\mathrm{SL}_2(\C)$ characters, with the same
factorization of the traceless Riley polynomial~\cite[Thms.~7.5 and 7.6]{NY}. The proof in
Section~\ref{sec:2bridge} is Lewallen's argument~\cite[proof of Prop.~2.3]{Lewallen} written in
quaternions; we include it because it locates the characters explicitly, which is what
Section~\ref{sec:2bridge-pillow} and Figure~\ref{fig:pillow} use.

\begin{theorem}[Klassen; Hedden--Herald--Kirk]\label{thm:B}
Let $n\ge2$ and $\gcd(3,n)=1$. The irreducible traceless $\SU$ characters of $T(3,n)$ are in bijection
with the traceless arcs of Section~\ref{sec:3n} (the arcs of irreducible characters that contain a
traceless point), and their number is
\[
N(3,n)=\lfloor 5n/6\rfloor-\lfloor n/6\rfloor=-\sigma(T(3,n))/2.
\]
Exactly $(\det-1)/2$ of them are binary dihedral: one if $n$ is even $(\det T(3,n)=3)$ and none if $n$
is odd $(\det T(3,n)=1)$.
\end{theorem}

The count $-\sigma/2$ of irreducible traceless characters of a torus knot is in~\cite[\S5.1]{HHK1}
and~\cite[\S9.5]{DSEquiv}, and every knot has exactly $(\det-1)/2$ irreducible binary-dihedral
characters~\cite[Thm.~10]{Klassen}. Section~\ref{sec:3n} derives the closed form by counting residues in
an interval and exhibits the dihedral character for even $n$.

\subsection{Gradings and the S-complex}\label{sec:intro-scx}
\emph{Conventions.} $T(3,n)$ denotes the positive torus knot, and signatures are
normalized so that positive torus knots have $\sigma<0$; for example $\sigma(T(2,3))=-2$ and
$\sigma(T(3,4))=-6$~\cite[(9.38)]{DSEquiv}. Chain complexes are homologically graded, with differentials
of degree $-1$, and $\Z/4$ grading indices are read modulo $4$. A $\Z/4$-graded group whose absolute
grading has been fixed is recorded by its rank vector $(r_0,r_1,r_2,r_3)$, the ranks over $\mathbb Q$ in
gradings $0,1,2,3$; $\Z^r_{(k)}$ denotes $\Z^r$ in grading $k$. Three normalizations of the $\Z/4$
instanton grading occur. Daemi--Scaduto's grading $\gr_{DS}$ puts the reducible generator of the S-complex (the summand $\Z_{(0)}$
below, given by the abelian traceless character) in degree $0$~\cite[(2.18), Def.~3.21]{DSEquiv}. The grading of $\Inat(K)$ differs from it by
$\sigma(K)\bmod4$~\cite[Thm.~8.9]{DSEquiv}. For $n$ odd, Poudel and Saveliev grade each irreducible flat
$\SO$ connection $\beta$ on $\Sig(2,3,n)$ by $\mu(\beta)\in\Z/4$, defined through an equivariant index on
the double branched cover~\cite[Prop.~6.5]{PS}; Proposition~\ref{prop:split} relates $\mu$ to
$\gr_{DS}$. Rank vectors are given in Daemi--Scaduto's normalization unless the shift is stated, and a graded group
is identified with its rank vector when only ranks matter. (Section~\ref{sec:grading} also uses the
Floer grading of $\Sig(2,3,n)$, and Section~\ref{sec:819} the pillowcase grading.)

Let $K$ be a torus knot. Its traceless character variety is finite and nondegenerate (each point is a
nondegenerate critical point of the unperturbed Chern--Simons functional)~\cite[\S9.5]{DSEquiv}. Let
$C_*(K)$ be Daemi--Scaduto's irreducible complex, freely generated by the irreducible traceless
characters, with differential $d$. Their S-complex~\cite[Def.~3.21]{DSEquiv} is
\[
\widetilde C_*(K)=C_*(K)\oplus C_{*-1}(K)\oplus\Z_{(0)},\qquad
\td=\begin{pmatrix} d&0&0\\ v&-d&\delta_2\\ \delta_1&0&0\end{pmatrix},
\]
with $v\colon C_*\to C_{*-2}$, $\delta_1\colon C_1\to\Z$ and $\delta_2\colon\Z\to C_2$. Daemi and Scaduto
call it the framed instanton chain complex~\cite[\S3.4]{DSEquiv}; it is distinct from framed instanton
homology $I^\#$. Its homology is isomorphic to $\Inat(K)$ after the shift by
$\sigma(K)\bmod4$~\cite[Thm.~8.9]{DSEquiv}. The S-complex is defined up to S-chain homotopy
equivalence; for the representative above, whose chain group is determined by $K$, the total rank of
$\td$ over $\mathbb Q$ equals $(\dim\widetilde C-\dim H_*(\widetilde C))/2$. Daemi and Scaduto's invariant
$h(K)\in\Z$ is a Frøyshov-type concordance invariant of the local equivalence class of $\widetilde C(K)$;
$h(K)>0$ if and only if some $\alpha\in C_*$ with $d\alpha=0$ has $\delta_1(\alpha)\neq0$
\cite[Prop.~4.15]{DSEquiv}.

\begin{theorem}[Daemi--Scaduto; Hedden--Herald--Kirk for $n\le38$]\label{thm:C}
Let $n\ge2$ with $\gcd(3,n)=1$.
\begin{enumerate}
\item[(i)] For $n$ odd, put $a=-\sigma(T(3,n))/4$. Then $\widetilde C(T(3,n))$ has rank vector
$(1+a,a,a,a)$, and the irreducible flat $\SO$ connections on $\Sig(2,3,n)$ split evenly between the
Poudel--Saveliev gradings $\mu=1$ and $\mu=3$.
\item[(ii)] The total rank of $\td$ on $\widetilde C(T(3,n))$ is $0$ for $n\equiv1,2\pmod 6$ and $1$ for
$n\equiv4,5\pmod 6$, and in each case it equals $h(T(3,n))$; for $n\equiv4,5\pmod 6$ the component
$\delta_1$ is nonzero.
\end{enumerate}
\end{theorem}

Part~(i) is the computation of the irreducible torus-knot complex in~\cite[Cor.~3]{DScaduto}, which
resolved the conjecture of Poudel and Saveliev~\cite[\S7.4]{PS}; the equivalence of the knot and
branched-cover formulations is noted in~\cite[\S9.5]{DSEquiv}, and the case $n\equiv1\pmod 6$ is due to
Anvari~\cite[Ex.~6.2]{Anvari}. In part~(ii), the vanishing for $n\equiv1,2\pmod6$ is
\cite[(9.36), Prop.~9.37]{DSEquiv}, and the cases $n=4,5$, including $\delta_1\neq0$, are
\cite[\S\S9.3--9.4]{DSEquiv}. Hedden--Herald--Kirk computed the whole pattern for $n\le38$, from the reduced Khovanov ranks, for the
Kronheimer--Mrowka chain complex with their holonomy perturbations, which has the same chain rank and
homology~\cite[\S12.4]{HHK1}. For $n\equiv4,5\pmod6$ with $n\ge10$, part~(ii) is deduced here: the rank
follows from Corollary~\ref{cor:torus}, and $\delta_1\neq0$ from the recursion for $h$
in~\cite[Prop.~6.11]{DScaduto}. Proposition~\ref{prop:even} proves part~(ii) for all $n$ at once.

\begin{corollary}[of Li--Ye and Daemi--Scaduto]\label{cor:torus}
For every torus knot $K$, the total rank of the differential of $\widetilde C(K)$ is
$\tfrac12\bigl(1+|\sigma(K)|-\|\Delta_K\|_1\bigr)$, where $\|\Delta_K\|_1$ is the sum of the absolute
values of the coefficients of the Alexander polynomial.
\end{corollary}

For general torus knots this rank is studied in~\cite{WuebbenIV}, which gives a floor-sum formula for it, a sharp
lower bound, and the classification of the torus knots for which the differential vanishes.

Section~\ref{sec:819} compares the rank-one differential of $8_{19}=T(3,4)$ with the pillowcase computation
of Hedden--Herald--Kirk, and Section~\ref{sec:open} states open questions.

\subsection{Strategy}
The proof of Theorem~\ref{thm:A} follows Lewallen. Any two traceless elements of $\SU$ lie on a common
great circle of the sphere of traceless elements, so a representation of a group generated by two
meridians is binary dihedral as soon as it is traceless; in the binary-dihedral normal form $\rho(a)=j$,
$\rho(b)=e^{i\eta}j$, the two-bridge relation reduces to $e^{ip\eta}=1$, which places the characters. Riley's degree bound then identifies
the traceless Riley polynomial.

Theorem~\ref{thm:B} turns on a trace obstruction. Writing $\pi_1(S^3\setminus T(3,n))=\langle x,y\mid
x^3=y^n\rangle$, the trace of $\rho(x)$ is $\pm1$, so $\rho(x)$ cannot lie in the coset $jS^1$ of
$\Pin=S^1\cup jS^1$; irreducibility then forces $\rho(y)$ into that coset, where every element is
traceless, and that is possible only when $n$ is even. The count reduces to counting residues modulo
$n$ in the interval $(n/6,5n/6)$.

For Theorem~\ref{thm:C}(i), Daemi and Scaduto's computation~\cite[Cor.~3]{DScaduto} gives $C_*=\Z^a_{(1)}\oplus\Z^a_{(3)}$, and
the decomposition $\widetilde C=C_*\oplus C_{*-1}\oplus\Z_{(0)}$ gives the rank vector. The two traceless
characters of $T(3,n)$ lying over a flat connection on $\Sig(2,3,n)$ have the same grading, and
$\mu=\gr_{DS}+\sigma$, with $\sigma\equiv0\pmod4$ for $n$ odd; so the knot statement and the cover statement
are equivalent. For part~(ii) and Corollary~\ref{cor:torus}, $\dim\widetilde C=1+|\sigma|$, $\dim
H_*(\widetilde C)=\operatorname{rank}\Inat=\|\Delta\|_1$, and $\dim H=\dim C-2\operatorname{rank}\td$. Li--Ye's
theorem gives $\operatorname{rank}\Inat=\|\Delta\|_1$ for all torus knots; for $T(3,n)$ we also obtain it by
squeezing $\Inat$ between Kronheimer--Mrowka's Alexander lower bound and the reduced Khovanov upper bound,
the latter from Sch\"utz's chain decomposition.

\subsection{Companion papers and computer checks}\label{sec:intro-companions}
Three companion papers continue this work. The rank in Corollary~\ref{cor:torus} is studied for all torus
knots in~\cite{WuebbenIV}, which also determines the ranks over $\mathbb Q$ of the graded components of the
differential. The pillowcase side of the family $T(3,n)$, $n\ge4$, is computed in~\cite{WuebbenV}, for the
decomposition of~\cite[\S10]{HHK2} given by the integers $(r,s)$, $3r+ns=1$, that Hedden, Herald and Kirk use
for $T(3,4)$, $T(3,5)$ and $T(3,7)$, and for sufficiently small holonomy perturbations in an explicit closed
cone of parameters. For these choices the pillowcase homology has the $\Z/4$-graded dimensions of $\Inat$,
and, after its grading is shifted by $\sigma$ and its closed components are suitably graded, the pillowcase
complex has the graded generators of $\widetilde C$ and a differential of the same rank. The pretzel knots
$P(-2,3,q)$, a family that begins with $T(3,4)$ and $T(3,5)$, are treated in~\cite{WuebbenII,WuebbenIIb}.

Several statements are also checked numerically, independently of the proofs: the traceless Riley
polynomials of Theorem~\ref{thm:A} for all two-bridge knots with $p\le9$, by exact computation of the
Riley word; the count of Proposition~\ref{prop:count} for $n\le25$, by enumerating arcs; and the
identity~\eqref{eq:N2a} and the grading split of Proposition~\ref{prop:split} for odd $n\le43$, by
evaluating~\eqref{eq:FS}--\eqref{eq:rho}. The programs are included as ancillary files with this preprint
and at \url{https://github.com/bwuebben/pillowcase-bounding-cochain}, a repository shared with the
companion papers~\cite{WuebbenII,WuebbenIIb,WuebbenIV,WuebbenV}. No proof depends on them.

\section{The pillowcase and traceless characters}\label{sec:pillow}

We work with $\SU$ as the unit quaternions. An element is traceless iff it is a purely imaginary unit
quaternion (eigenvalues $\pm i$, squaring to $-1$); as a rotation it is a rotation by $\pi$. The
traceless elements form a $2$-sphere $S\subset\operatorname{Im}\HH\cong\R^3$.

\begin{proposition}[Classical; cf.~\cite{Lin},~{\cite[Prop.~3.1]{HHK1}}]\label{prop:pillow}
The traceless $\SU$ character variety of the four-punctured sphere is the pillowcase $\pill$, with four
$\Z/2$ orbifold points at $(\gamma,\theta)\in\{0,\pi\}^2$ and the symplectic form descending from
$d\gamma\wedge d\theta$.
\end{proposition}

\begin{proof}[Sketch]
Write $\pi_1$ of the four-punctured sphere as $\langle X_1,\dots,X_4\mid X_1X_2X_3X_4=1\rangle$, and write
$X_i$ also for its image, an element of $S$. Since purely imaginary units satisfy $X_i^{-1}=-X_i$, the
relation gives $X_1X_2=X_4X_3=:A$. For pure imaginary units, $X_iX_j=-\langle X_i,X_j\rangle+X_i\times
X_j$, so when $A\neq\pm1$ its axis is well defined; equating the two expressions for $A$ forces the four
$X_i$ to be coplanar, orthogonal to $\operatorname{axis}(A)$. Placing that axis along $k$ and writing
$X_m=\cos t_m\,i+\sin t_m\,j$, the relation becomes $t_2-t_1=t_3-t_4$; gauge-fixing $t_1=0$ and setting
$\gamma=t_2$, $\theta=t_4$ leaves two parameters, and the residual conjugation by the $\pi$-rotation about
$i$ acts by $(\gamma,\theta)\mapsto(-\gamma,-\theta)$. The four fixed points are the reducible (collinear)
characters.
\end{proof}

\begin{figure}[ht]
\centering
\begin{tikzpicture}[scale=1.1]
  \draw[thick] (0,0) rectangle (2,4);
  \foreach \x/\y in {0/0,2/0,0/2,2/2,0/4,2/4}{\fill (\x,\y) circle (2.2pt);}
  \node[below left] at (0,0) {\footnotesize $(0,0)$};
  \node[below right] at (2,0) {\footnotesize $(\pi,0)$};
  \node[left] at (0,2) {\footnotesize $(0,\pi)$};
  \node[right] at (2,2) {\footnotesize $(\pi,\pi)$};
  \node[above left] at (0,4) {\footnotesize $(0,2\pi)$};
  \node[above right] at (2,4) {\footnotesize $(\pi,2\pi)$};
  \draw[->] (-0.15,0)--(2.4,0) node[right] {\footnotesize $\gamma$};
  \draw[->] (0,-0.15)--(0,4.4) node[above] {\footnotesize $\theta$};
  \draw[blue,thick] (0,0)--(2,2);
  \draw[red,thick] (0,4)--(2,0);
  \node[blue] at (0.55,1.0) {\footnotesize $\ell_0$};
  \node[red] at (0.85,2.6) {\footnotesize $L_1$};
\end{tikzpicture}
\caption{The pillowcase $\pill$ drawn as the rectangle $[0,\pi]\times[0,2\pi]$: each vertical edge is
folded at its midpoint, $(\gamma,\theta)\sim(\gamma,2\pi-\theta)$ for $\gamma\in\{0,\pi\}$, and the
horizontal edges are identified, so $(0,2\pi)=(0,0)$ and $(\pi,2\pi)=(\pi,0)$. The four orbifold points
are $(0,0)$, $(\pi,0)$, $(0,\pi)$ and $(\pi,\pi)$. For the trefoil $\bp(3,1)$ the trivial tangle gives the
diagonal arc $\ell_0=\{\gamma=\theta\}$ and the rational tangle the straight arc $L_1=\{\theta=2\pi-2\gamma\}$;
they meet at the corner $(0,0)$ and at the interior point $(2\pi/3,2\pi/3)$
(Section~\ref{sec:2bridge-pillow}). The earring correspondence replaces $\ell_0$ by a figure-eight curve
$L_0$ near it, whose single self-crossing lies near the midpoint of $\ell_0$ and whose two lobes encircle
its end corners (Remark~\ref{rmk:earring}).}
\label{fig:pillow}
\end{figure}

For a knot $K$ with a two-tangle decomposition, a traceless representation of $\pi_1(S^3\setminus K)$
restricts to a compatible pair of tangle representations and so determines a point of $L_0\cap L_1$; for
two-bridge knots this gives a bijection between the intersection points and the traceless characters of
the knot group~\cite[\S10]{HHK1}, which after perturbation give the generators of the singular instanton
chain complex~\cite{KM1,KM2}. A two-bridge knot group is generated by two meridians $a,b$, and an
irreducible traceless representation is determined up to conjugacy by the single invariant
$\cos\omega=\langle\rho(a),\rho(b)\rangle$, the cosine of the angle between $\rho(a),\rho(b)\in S$.

\section{Two-bridge knots}\label{sec:2bridge}

\subsection{Traceless characters are binary dihedral}
The group of $\bp(p,q)$, with $q$ odd, is $\langle a,b\mid wa=bw\rangle$ with
$w=a^{\varepsilon_1}b^{\varepsilon_2}a^{\varepsilon_3}\cdots b^{\varepsilon_{p-1}}$ and
$\varepsilon_i=(-1)^{\lfloor iq/p\rfloor}$, where $a,b$ are meridians~\cite{Riley}. Riley's
parametrization~\cite{Riley} of the nonabelian $\mathrm{SL}_2(\C)$ representations with meridian
eigenvalue $s$ is, after a diagonal conjugation,
\[
\rho(a)=\begin{pmatrix}s&1\\0&s^{-1}\end{pmatrix},\qquad
\rho(b)=\begin{pmatrix}s&0\\-u&s^{-1}\end{pmatrix},
\]
cut out by a polynomial $\Phi_{p,q}(s,u)$ of degree $(p-1)/2$ in $u$; the traceless locus is $s=i$. On
that locus $\tr\rho(ab)=s^2+s^{-2}-u=-2-u$, and for traceless $\rho(a),\rho(b)\in S$ one has
$\tr\rho(ab)=2\operatorname{Re}\bigl(\rho(a)\rho(b)\bigr)=-2\cos\omega$. Hence Riley's parameter is
$u=-2-\tr\rho(ab)=2\cos\omega-2$. Write $\phi_p$ for the monic normalization of $\Phi_{p,q}(i,\cdot)$, the
traceless Riley polynomial; Theorem~\ref{thm:A} shows it is independent of $q$.

\begin{proof}[Proof of Theorem~\ref{thm:A}]
Let $\rho$ be traceless. The unit imaginary quaternions $\rho(a),\rho(b)$ lie in a plane through the
origin, so after conjugation $\rho(a)=j$ and $\rho(b)=e^{i\eta}j=\cos\eta\,j+\sin\eta\,k$. Both lie
in the coset $jS^1$ of $\Pin=S^1\cup jS^1$, where $S^1=\{e^{it}\}$, so $\rho$ is binary dihedral. Since
$a^{\pm1}\mapsto\pm j$, $b^{\pm1}\mapsto\pm e^{i\eta}j$ and $je^{i\eta}j=-e^{-i\eta}$, each of the
$(p-1)/2$ consecutive pairs $a^{\varepsilon}b^{\varepsilon'}$ in $w$ maps to $\pm e^{-i\eta}$, so
$\rho(w)=\pm e^{-i(p-1)\eta/2}$ and
\[
\rho(w)\rho(a)\rho(w)^{-1}=e^{-i(p-1)\eta}j.
\]
The relation $wa=bw$ therefore holds iff $e^{-i(p-1)\eta}=e^{i\eta}$, that is, iff
$\eta=2\pi k/p$ for an integer $k$. Here $\cos\omega=\langle j,e^{i\eta}j\rangle=\cos\eta$;
$k=0$ gives the abelian representation, and conjugation by $j$ sends $\eta$ to $-\eta$. The
irreducible traceless characters are therefore exactly those with $k=1,\dots,(p-1)/2$, with distinct
values $\cos\omega=\cos(2\pi k/p)$ that depend only on $p$.

Each of them is a nonabelian representation with $s=i$, so $\Phi_{p,q}(i,u_k)=0$ at
$u_k=2\cos(2\pi k/p)-2=-4\sin^2(\pi k/p)$. Since $\Phi_{p,q}(i,\cdot)$ is a nonzero polynomial of degree
at most $(p-1)/2$~\cite{Riley},~\cite[Thm.~7.6]{NY}, it has no other roots, and
$\phi_p(u)=\prod_k(u-u_k)$, as in~\cite[Thm.~7.6]{NY}. Its constant term is $\prod_k 4\sin^2(\pi k/p)=p$
and the sum of its roots is $\sum_k(2\cos(2\pi k/p)-2)=-1-(p-1)=-p$. The Galois group of the $p$th
cyclotomic field permutes the roots $u_k$, so $\phi_p$ has rational coefficients; they are algebraic
integers because the roots are, and hence lie in $\Z$.
\end{proof}

\begin{remark}
Small cases: $\phi_3=u+3$, $\phi_5=u^2+5u+5$, $\phi_7=u^3+7u^2+14u+7$, $\phi_9=u^4+9u^3+27u^2+30u+9$. For
$p$ prime, $\phi_p(x-2)$ is the minimal polynomial of $2\cos(2\pi/p)$, irreducible of degree $(p-1)/2$.
The two two-bridge knots of determinant $9$, the torus knot $9_1=\bp(9,1)$ and the hyperbolic knot
$6_1=\bp(9,2)$, have the same traceless Riley polynomial $\phi_9$.
\end{remark}

\subsection{The pillowcase picture and the earring}\label{sec:2bridge-pillow}
The traceless character variety of a rational tangle is a straight arc in $\pill$ joining two corners, of
rational slope determined by the tangle~\cite[Thm.~3]{HHK1}. Under the branched double cover
$(\R/2\pi\Z)^2\to\pill$ these arcs lift to straight lines~\cite{Lin,HHK1,CHKK}. For $\bp(p,q)$ the two
lifted lines have directions $(1,1)$ and $(q,q-p)$ in the coordinates $(\gamma,\theta)$~\cite[Thm.~3]{HHK1}
and meet in $|\det\left(\begin{smallmatrix}1&1\\q&q-p\end{smallmatrix}\right)|=p$ points; correspondingly,
$\Sigma_2(K)=L(p,q)$ is the union of two solid tori, the double branched covers of the two rational
tangles.

The naive count on $\pill$ itself is $(p+1)/2$: the $(p-1)/2$ interior dihedral characters of
Theorem~\ref{thm:A} together with the single abelian character at a corner. The passage to $\det=p$
comes from the \emph{earring}: the reduced theory $\Inat$ is defined by adding a small earring to the
knot~\cite[\S4.3]{KM2}, and in the pillowcase picture this replaces the curve of the tangle that carries
the earring by its image under the earring correspondence $E$~\cite{CHKK,HK}. For the trivial tangle the
image of the diagonal arc $\ell_0$ is a figure-eight curve $L_0=E(\ell_0)$ close to it~\cite[Thm.~1]{HHK1}.
Each of the $(p-1)/2$ interior intersection points of the straight arc $L_1$ with $\ell_0$ gives two
intersection points with $L_0$, and the corner gives one, for $p$ in all. Since $L_1$ is straight and
$L_0$ is close to $\ell_0$, there are no immersed bigons. Hedden--Herald--Kirk show that for every two-bridge knot the
differentials vanish, the pillowcase homology has rank $\det K$, and the relative $\Z/4$ gradings agree
with those of $\Inat$~\cite[\S7]{HHK2}, with the grading identification outlined in~\cite[\S6.1]{HHK2}.
Theorem~\ref{thm:A} adds the identification of the interior generators with the binary-dihedral
characters.

\begin{remark}[The earring correspondence]\label{rmk:earring}
Herald and Kirk prove, at the level of regular homotopy and in the limit of small perturbation, that $E$
replaces an immersed arc joining two corners (satisfying a mild genericity condition, `good' in the sense
of~\cite{HK}) by a figure-eight curve near it, an immersed circle with a single transverse self-crossing
near the midpoint of the arc whose two lobes encircle the end corners, and replaces an immersed circle
avoiding the corners by two parallel copies~\cite[Thm.~B, \S6.4.2, Thms.~6.8 and 6.15]{HK}. Thus $E\circ L$
is immersed even when $L$ is embedded. For a two-bridge knot the only self-crossing at which a bounding
cochain could sit is the one near the midpoint of the arc, and Smith shows that the cochain associated to
a rational tangle with an earring vanishes if~\cite[Conj.~D]{CHKK} holds~\cite[Lemma~5.3]{Smith}. Off the
two-bridge family there is evidence that some decompositions require nonzero bounding cochains: for the
$(4,5)$-torus knot, Cazassus--Herald--Kirk--Kotelskiy show that with zero cochains the Floer homology
depends on which tangle of the decomposition of~\cite[\S11.8]{HHK2} carries the earring, and argue
informally that a nonzero cochain restores agreement~\cite[\S10.3]{CHKK}; Smith proves that if their
Conjecture~D holds, then some tangle without earring must carry a nonzero cochain, exhibited in a
decomposition of $P(-2,3,5)$~\cite[Thm.~5.9]{Smith}. For a second decomposition of $T(3,4)$
in the family of~\cite[\S10]{HHK2}, a numerical computation gives pillowcase homology of rank $7$ with zero
cochains, against $\operatorname{rank}\Inat(T(3,4))=5$, and rank $5$ with a nonzero cochain at a
self-intersection of the arc~\cite[\S7]{WuebbenV}. For Smith's decomposition of $P(-2,3,7)$, the pairing of
the earring curve with a piecewise-linear model of Smith's curve, computed with zero cochain in an algebraic
model of the wrapped Fukaya category, has dimension $7$, against $\operatorname{rank}\Inat(P(-2,3,7))=9$, and
Maurer--Cartan deformations of the model raise it to $9$~\cite{WuebbenIIb}.
\end{remark}

\section{The \texorpdfstring{$(3,n)$}{(3,n)}-torus knots}\label{sec:3n}

\subsection{The representation arcs}
Throughout this section $n\ge2$ and $\gcd(3,n)=1$. Write
$\pi_1(S^3\setminus T(3,n))=\langle x,y\mid x^3=y^n\rangle$. The element
$z=x^3=y^n$ is central, so any irreducible $\SU$ representation sends $z\mapsto\pm1$ and
\[
\rho(x)=e^{\vartheta_1\hat u},\quad\rho(y)=e^{\vartheta_2\hat v},\qquad
\vartheta_1=\tfrac{\pi\ell_1}{3},\ \ \vartheta_2=\tfrac{\pi\ell_2}{n},
\]
with $\hat u,\hat v$ unit imaginary quaternions, $\ell_1\in\{1,2\}$, $0<\ell_2<n$, and
$\ell_1\equiv\ell_2\pmod2$, since $\rho(x)^3=\rho(y)^n=\rho(z)$. The angle $\psi=\angle(\hat u,\hat
v)\in(0,\pi)$ is the only remaining modulus; each such pair $(\ell_1,\ell_2)$ is a one-parameter arc of
irreducible characters~\cite[Thm.~1]{Klassen}. Call an arc a \emph{traceless arc} if it contains a
traceless character.

\subsection{The traceless locus}
Fix integers $c,f$ with $cn+3f=1$. The meridian of $T(3,n)$ is $x^cy^f$~\cite[\S5.1]{HHK1}; it does not
depend on the choice of $(c,f)$, since $x^{c+3t}y^{f-nt}=x^cz^ty^fz^{-t}=x^cy^f$. A direct quaternion
computation gives
\[
\tfrac12\tr\rho(x^cy^f)=\operatorname{Re}\rho(x^cy^f)
=\cos(c\vartheta_1)\cos(f\vartheta_2)-\sin(c\vartheta_1)\sin(f\vartheta_2)\cos\psi,
\]
which is affine in $\cos\psi$. Hence each arc contains exactly one traceless character, at
\begin{equation}\label{eq:tf}
\cos\psi=\cot(c\vartheta_1)\,\cot(f\vartheta_2),
\end{equation}
provided the right side lies in $(-1,1)$; otherwise the arc carries none, and the boundary case never
occurs, as the proof below shows. By construction $\rho(x)^3=\rho(y)^n=(-1)^{\ell_1}$.

\begin{proposition}[Count]\label{prop:count}
There are
\[
N(3,n)=\lfloor 5n/6\rfloor-\lfloor n/6\rfloor=\begin{cases}
2(n-1)/3,& n\equiv1\pmod 6,\\ (2n-1)/3,& n\equiv2\pmod 6,\\
(2n+1)/3,& n\equiv4\pmod 6,\\ 2(n+1)/3,& n\equiv5\pmod 6
\end{cases}
\]
irreducible traceless $\SU$ characters of $T(3,n)$, and in addition exactly one abelian traceless
character. Moreover $N(3,n)=-\sigma(T(3,n))/2$~\cite[\S5.1]{HHK1},~\cite[\S9.5]{DSEquiv}, and for $n$ odd
half of the irreducible traceless characters have $\rho(z)=1$ and half have $\rho(z)=-1$.
\end{proposition}

\begin{proof}
From $cn+3f=1$ we get $\gcd(c,3)=\gcd(f,n)=1$. Since $\ell_1\in\{1,2\}$ and $3\nmid c\ell_1$, the first
cotangent in~\eqref{eq:tf} always has $|\cot(c\vartheta_1)|=\cot(\pi/3)=1/\sqrt3$; so an arc is a
traceless arc iff $|\cot(f\vartheta_2)|<\sqrt3$, iff $\pi f\ell_2/n$ lies in $(\pi/6,5\pi/6)$ modulo $\pi$,
iff the residue $\lambda=f\ell_2\bmod n$ lies in the open interval $(n/6,\,5n/6)$. Equality never occurs:
it would force $6\mid n$, impossible with $\gcd(3,n)=1$, so the boundary case $\cos\psi=\pm1$ does not
arise. The parity condition assigns to each $\ell_2\in\{1,\dots,n-1\}$ exactly one $\ell_1$ ($\ell_1=1$ for
$\ell_2$ odd, $\ell_1=2$ for $\ell_2$ even), and $\ell_2\mapsto\lambda=f\ell_2\bmod n$ is a bijection of
$\{1,\dots,n-1\}$; hence
\[
N(3,n)=\#\{\lambda\in\Z:\ n/6<\lambda<5n/6\}=\lfloor 5n/6\rfloor-\lfloor n/6\rfloor,
\]
which evaluates, in the four residues $n\equiv1,2,4,5\pmod6$, to the stated cases. Distinct arcs give
distinct characters, since $\tr\rho(y)=2\cos(\pi\ell_2/n)$ separates them, and the abelian traceless
character (every meridian to a fixed traceless element) is unique up to conjugacy. Finally, for $n$ odd
the map $(\ell_1,\ell_2)\mapsto(3-\ell_1,n-\ell_2)$ respects the parity condition and exchanges
$\rho(z)=-1$ with $\rho(z)=1$; it preserves the condition $\lambda\in(n/6,5n/6)$, since $\lambda\mapsto
n-\lambda$, and it preserves $\psi$ in~\eqref{eq:tf}, since both cotangents change sign. Being
fixed-point-free, it splits the $N(3,n)$ traceless characters evenly between the two signs of $\rho(z)$.
\end{proof}

\subsection{The dihedral dichotomy}
\begin{proof}[Proof of Theorem~\ref{thm:B}]
The bijection and the count are Proposition~\ref{prop:count}. Conjugate a binary-dihedral image into
$\Pin=S^1\cup jS^1$. Since $\tr\rho(x)=2\cos(\pi\ell_1/3)=\pm1$, the element $\rho(x)$ cannot lie in the
coset $jS^1$, whose elements are traceless; it therefore lies in $S^1$. Irreducibility forces $\rho(y)$ to
lie in $jS^1$. Thus $\vartheta_2=\pi/2$, equivalently $\ell_2=n/2$, so $n$ must be even.

Conversely, suppose $n$ is even and take $\ell_2=n/2$. The parity condition determines a unique
$\ell_1\in\{1,2\}$. In the relation $cn+3f=1$ the integer $f$ is odd, and hence
$\cot(f\vartheta_2)=\cot(f\pi/2)=0$. Equation~\eqref{eq:tf} gives $\cos\psi=0$. Relative to the circle
$S^1$ containing $\rho(x)$, this places $\rho(y)$ in $jS^1$; the resulting traceless character is
irreducible and binary dihedral. Hence there is exactly one such character for even $n$ and none for odd
$n$. This equals $(\det T(3,n)-1)/2$, since $\det T(3,n)=|\Delta_{T(3,n)}(-1)|$ is $3$ for even $n$ and $1$
for odd $n$ (Lemma~\ref{lem:homrank} below), in agreement with~\cite[Thm.~10]{Klassen}.
\end{proof}

Theorem~\ref{thm:B} contrasts with Theorem~\ref{thm:A}: the irreducible traceless characters of a
two-bridge knot are all binary dihedral, those of $T(3,n)$ with $n$ odd none. In the pillowcase pictures
of~\cite[\S\S11.6--11.7]{HHK2} for $T(3,4)$ and $T(3,5)$ the non-dihedral characters lie on straight arcs
of a singular tangle character variety, and it is the holonomy perturbation resolving that variety which
produces non-straight curves and, with them, bigons (Section~\ref{sec:819}).

\section{Gradings and the S-complex differential}\label{sec:grading}

The double branched cover of $T(3,n)$ is the Brieskorn manifold $\Sig(2,3,n)$~\cite[\S7.4]{PS}. For odd $n$
it is an integral homology sphere; its irreducible flat $\SO$ connections have rotation numbers
$(m_1,m_2,m_3)=(1,m_2,m_3)$ indexing the three singular fibers, with $m_2,m_3$ even, $0<m_2<3$,
$0<m_3<n$. Since $2$ is the only even value in $(0,3)$, necessarily $m_2=2$, and $m_3$ ranges over the
even integers in $(0,n)$ meeting the spherical-triangle inequalities on the angles
$(\pi/2,2\pi/3,\pi m_3/n)$. There are $a=-\sigma(T(3,n))/4$ such connections, all
nondegenerate~\cite{FS},~\cite[\S7.4]{PS}, and all invariant under the covering
involution~\cite[\S7.4]{PS},~\cite[\S9.5]{DSEquiv}. A traceless character $\alpha$ of $T(3,n)$ \emph{lies
over} a flat connection $\beta$ on $\Sig(2,3,n)$ if the adjoint of $\alpha$ pulls back to $\beta$. Over each
irreducible $\beta$ lie exactly two irreducible traceless characters, exchanged by the \emph{flip} $\iota$,
which twists a representation by the sign character sending each meridian to
$-1$~\cite[Prop.~4.1(c)]{PS},~\cite[\S\S2.3, 9.1]{DSEquiv}. Consequently
\begin{equation}\label{eq:N2a}
N(3,n)=2a\qquad(n\text{ odd}).
\end{equation}

\subsection{The grading formula}
The proof of Proposition~\ref{prop:split} does not use the formulas of this subsection.
In Poudel and Saveliev's description of $\Inat$ through equivariant gauge theory on the double branched
cover, each irreducible
$\beta$ contributes a pair of circles of critical points, exchanged by the flip~\cite[Prop.~4.1(c)]{PS}.
Both circles are graded by $\mu(\beta)$~\cite[Prop.~6.5]{PS}, each contributes two generators, in
gradings $\mu(\beta)$ and $\mu(\beta)+1$, and the trivial connection $\Theta$ contributes one generator, in
grading $\sigma\bmod4$. Anvari evaluates $\mu$ explicitly~\cite[\S6]{Anvari}:
\begin{equation}\label{eq:mu}
\mu(\beta)=\tfrac12\gr_{FS}(\beta)+\tfrac14\bigl(1-\rho_{\mathrm{Ad}\,\beta}\bigr)\pmod 4,
\end{equation}
where $\gr_{FS}(\beta)$ is an integer whose reduction mod $8$ is the Floer grading of $\beta$ in the
instanton homology of $\Sig(2,3,n)$~\cite{FS}, and $\rho_{\mathrm{Ad}\,\beta}$ is the equivariant
$\rho$-invariant. In Anvari's normalization~\cite[(3), Thm.~B]{Anvari},
\begin{align}
\gr_{FS}(\beta)&=\frac{e^2}{3n}+\frac23\sum_{k=1}^{2}\cot\tfrac{\pi k}{3}\cot\tfrac{\pi b_2 k}{3}
\sin^2\tfrac{\pi k m_2}{3}+\frac2n\sum_{k=1}^{n-1}\cot\tfrac{\pi k}{n}\cot\tfrac{\pi b_3 k}{n}
\sin^2\tfrac{\pi k m_3}{n},\label{eq:FS}\\
\rho_{\mathrm{Ad}\,\beta}&=1-\frac{2}{3n}-\frac43\sum_{k=1}^{2}\cot\tfrac{\pi k}{3}\tan\tfrac{\pi b_2 k}{3}
\cos^2\tfrac{\pi k m_2}{3}-\frac4n\sum_{k=1}^{n-1}\cot\tfrac{\pi k}{n}\tan\tfrac{\pi b_3 k}{n}
\cos^2\tfrac{\pi k m_3}{n},\label{eq:rho}
\end{align}
with $e=3n\,m_1+2n\,m_2+6\,m_3$, where $b_1,b_2,b_3$ are Seifert invariants of $\Sig(2,3,n)$, integers with
$3nb_1+2nb_2+6b_3=1$ and $b_2,b_3$ even; since
\eqref{eq:FS}--\eqref{eq:rho} depend on $b_2$ only modulo $3$ and on $b_3$ only modulo $n$, one may take
$b_2\equiv(2n)^{-1}\pmod3$ and $b_3\equiv6^{-1}\pmod n$. These formulas apply only for $n$ odd, where $\Sig(2,3,n)$ is a
homology sphere; the rank argument of Section~\ref{sec:grading-diff} does not use the parity of $n$. By
Proposition~\ref{prop:split}, for each odd $n$ the split of Theorem~\ref{thm:C}(i) is an identity among
the trigonometric sums~\eqref{eq:FS}--\eqref{eq:rho}: the values~\eqref{eq:mu} are $1$ and $3$ equally
often. Daemi and Scaduto note this consequence of their computation for Anvari's
formulas~\cite{DScaduto}. The $\rho$-invariant term is needed: on $\Sig(2,3,11)$ the four connections
have Floer gradings $1,7,5,3\pmod8$ and $\rho$-invariants $-1,3,-1,3$, so with a constant $\rho$-term
the values~\eqref{eq:mu} would be four distinct residues.

\begin{example}\label{ex:237}
For $\Sig(2,3,7)$ the two connections $(1,2,2)$ and $(1,2,4)$ have $\gr_{FS}\equiv7$ and $3\pmod8$ and
$\rho_{\mathrm{Ad}\,\beta}=3$ for both, hence $\mu=3$ and $\mu=1$, as in~\cite[Ex.~6.1]{Anvari}.
\end{example}

\subsection{The homology ranks}

\begin{lemma}[Li--Ye; a second proof via Khovanov homology]\label{lem:homrank}
For every $n\ge1$ coprime to $3$,
\[
\operatorname{rank}_{\mathbb Q}\Inat(T(3,n);\mathbb Q)=\|\Delta_{T(3,n)}\|_1=
\begin{cases}4k+1,&n=3k+1,\\ 4k+3,&n=3k+2,\end{cases}
\]
and $|\Delta_{T(3,n)}(-1)|$ is $1$ for $n$ odd and $3$ for $n$ even.
\end{lemma}

The equality $\operatorname{rank}\Inat=\|\Delta\|_1$ is the case of $T(3,n)$ in Li--Ye's computation of
the sutured instanton knot homology $KHI$ of Kronheimer and Mrowka for torus knots~\cite[Cor.~1.8]{LiYe},
combined with $\Inat(K;\mathbb Q)\cong KHI(K;\mathbb Q)$~\cite[Prop.~1.4]{KM2}; Hedden--Herald--Kirk checked it for $n\le38$~\cite[\S12.4]{HHK1},
and the cases $n=4,5$ are in~\cite[\S\S9.3--9.4]{DSEquiv}. The following proof replaces Li--Ye's
computation by Sch\"utz's computation of reduced Khovanov homology, combined with the
Kronheimer--Mrowka bounds.

\begin{proof}
Sch\"utz works over the Bar--Natan ring $\Z[X,Y]/(X^2-XY)$ (he writes $h$ for $Y$; this $h$ is unrelated to the invariant $h(K)$) and writes $A$ for its
$q$-shifted rank-one free module and $A(j)$ for the two-term complex with differential $(2X-Y)^j$
\cite[Def.~3.2]{Schutz}; the reduced Khovanov complex is obtained by tensoring the Bar--Natan complex
with $\Z$, where $X$ and $Y$ act by zero~\cite[\S3]{Schutz}. His chain decompositions give, with grading
shifts suppressed,
\begin{align*}
C_{BN}(T(3,3k+1))&\simeq A\oplus A(1)^k\oplus A(2)^k,\\
C_{BN}(T(3,3k+2))&\simeq A\oplus A(1)^{k+1}\oplus A(2)^k
\end{align*}
\cite[Thm.~5.3(2),(4)]{Schutz}. After setting $X=Y=0$, $A$ contributes one generator and each $A(j)$ a complex on two generators
with zero differential, so the reduced complex is chain homotopy equivalent to a complex with $1+2k+2k=4k+1$,
respectively $1+2(k+1)+2k=4k+3$, generators and zero differential. Consequently
\begin{equation}\label{eq:kh-ranks}
\operatorname{rank}_{\mathbb Q}\widetilde{Kh}(T(3,n);\mathbb Q)=
\begin{cases}
4k+1,&n=3k+1,\\
4k+3,&n=3k+2.
\end{cases}
\end{equation}

Up to multiplication by a monomial, the Alexander polynomial is
\[
\Delta_{T(3,n)}(t)=\frac{1+t^n+t^{2n}}{1+t+t^2},
\]
so $\Delta_{T(3,n)}(-1)=2+(-1)^n$. Below degree $n$ its coefficients repeat $1,-1,0$, and the quotient is
symmetric of degree $2n-2$. The central coefficient, in degree $n-1$, has absolute value one. If $n=3k+1$,
the coefficients in degrees below $n-1$ include $2k$ nonzero ones, each $\pm1$; if $n=3k+2$, they include
$2k+1$. By the symmetry of $\Delta$, $\|\Delta_{T(3,n)}\|_1=2\cdot2k+1=4k+1$, respectively
$2(2k+1)+1=4k+3$, the right side of~\eqref{eq:kh-ranks}. The group $KHI(K)$ is Alexander-graded with graded Euler characteristic $\pm\Delta_K(t)$
\cite[Thm.~1.1]{KMAlex}, so $\operatorname{rank}KHI(K)\ge\|\Delta_K\|_1$~\cite[Cor.~1.2]{KMAlex}. The
reduced Khovanov-to-instanton spectral sequence gives
$\operatorname{rank}\widetilde{Kh}(K)\ge\operatorname{rank}\Inat(K)$~\cite[Prop.~1.2, Cor.~1.3]{KM2}.
Together with $\Inat(K;\mathbb Q)\cong KHI(K;\mathbb Q)$~\cite[Prop.~1.4]{KM2}, this yields
\[
\|\Delta_{T(3,n)}\|_1\le \operatorname{rank}_{\mathbb Q}\Inat(T(3,n);\mathbb Q)
\le \operatorname{rank}_{\mathbb Q}\widetilde{Kh}(T(3,n);\mathbb Q)=\|\Delta_{T(3,n)}\|_1,
\]
and equality holds throughout.
\end{proof}

\subsection{The grading split for \texorpdfstring{$n$}{n} odd}

\begin{proposition}[Daemi--Scaduto]\label{prop:split}
Let $n$ be odd and coprime to $3$, and set $a=-\sigma(T(3,n))/4=N(3,n)/2$. Then $\widetilde C(T(3,n))$ has
rank vector $(1+a,a,a,a)$, and $a/2$ of the irreducible flat $\SO$ connections on $\Sig(2,3,n)$ have
$\mu=1$ and $a/2$ have $\mu=3$.
\end{proposition}

\begin{proof}
Daemi and Scaduto prove that for every torus knot $K$ the irreducible complex is
$C_*(K)\cong\Z_{(1)}^{\lceil-\sigma(K)/4\rceil}\oplus\Z_{(3)}^{\lfloor-\sigma(K)/4\rfloor}$, with zero
differential~\cite[Cor.~3]{DScaduto},~\cite[\S9.5]{DSEquiv}. For $n$ odd, $-\sigma/4=a$, so
$C_*=\Z^a_{(1)}\oplus\Z^a_{(3)}$, and $\widetilde C=C_*\oplus C_{*-1}\oplus\Z_{(0)}$ has rank vector
$(1+a,a,a,a)$; this is the rank vector conjectured by Poudel and Saveliev~\cite[\S7.4]{PS}.

Let $\beta$ be an irreducible flat connection on $\Sig(2,3,n)$ and $\alpha,\iota\alpha$ the two characters
lying over it. They have the same grading~\cite[Lemma~6.4 and the paragraph following it]{PS}; this also
follows from the flip symmetry $M(\alpha,\alpha')\to M(\iota\alpha,\iota\alpha')$ of the moduli
spaces~\cite[\S9.1]{DSEquiv}. Poudel and Saveliev define
$\mu(\beta)=\operatorname{ind}D^\tau_B(\beta,\Theta)+\sigma\pmod4$, where $\operatorname{ind}D^\tau_B$ is
the index of the deformation operator on deck-invariant sections over
$\R\times\Sig(2,3,n)$~\cite[Prop.~6.5]{PS}. This equivariant index equals the orbifold index over
$\R\times(S^3,T(3,n))$~\cite[\S3.5, (7)]{PS}, which is $\gr_{DS}(\alpha)$~\cite[(2.11), (2.18)]{DSEquiv};
both use weighted spaces with decay at the reducible end. Hence
\[
\mu(\beta)=\gr_{DS}(\alpha)+\sigma(T(3,n))\pmod4,
\]
consistent with the degree $\sigma\bmod4$ of the isomorphism of~\cite[Thm.~8.9]{DSEquiv}; Daemi and
Scaduto use the same identification of gradings in~\cite[\S9.4]{DSEquiv}. For $n$ odd,
$\sigma=-2N(3,n)=-4a\equiv0\pmod4$ by~\eqref{eq:N2a}, so each $\beta$ contributes two generators of $C_*$
in degree $\mu(\beta)$. If $\nu_j$ denotes the number of $\beta$ with $\mu(\beta)=j$, the ranks of $C_*$
give $\nu_1=\nu_3=a/2$ and $\nu_0=\nu_2=0$. In particular $a$ is even; explicitly $a=(n-1)/3$ for
$n\equiv1\pmod6$ and $a=(n+1)/3$ for $n\equiv5\pmod6$.
\end{proof}

\subsection{The S-complex differential}\label{sec:grading-diff}

\begin{proposition}[Daemi--Scaduto for $n\equiv1,2\pmod6$ and $n=4,5$]\label{prop:even}
Let $n\ge2$ be coprime to $3$ and $N=N(3,n)$. The S-complex $\widetilde C(T(3,n))$ has rank $2N+1$, and
the total rank of its differential $\td$ is
\[
\begin{array}{c|c|c|c}
n\bmod 6 & N & \operatorname{rank}_{\mathbb Q}\Inat(T(3,n);\mathbb Q) &
\operatorname{rank}_{\mathbb Q}\td\\ \hline
1 & 2(n-1)/3 & 2N+1 & 0\\
2 & (2n-1)/3 & 2N+1 & 0\\
4 & (2n+1)/3 & 2N-1 & 1\\
5 & 2(n+1)/3 & 2N-1 & 1.
\end{array}
\]
In each case $\operatorname{rank}\td=h(T(3,n))$, and for $n\equiv4,5\pmod6$ the component $\delta_1$ of
$\td$ is nonzero.
\end{proposition}

\begin{proof}
The irreducible complex has one generator for every irreducible traceless character, and these are
nondegenerate~\cite[\S9.5]{DSEquiv}; so $\widetilde C=C_*\oplus C_{*-1}\oplus\Z_{(0)}$ has rank $2N+1$,
with $N$ from Proposition~\ref{prop:count}. Lemma~\ref{lem:homrank} gives the homology ranks in the table
($n=6m+1,6m+2,6m+4,6m+5$ give $\|\Delta\|_1=8m+1,8m+3,8m+5,8m+7$), and the identity $\dim H=\dim
C-2\operatorname{rank}\td$ applied to $\widetilde C$, whose homology has the rank of
$\Inat$~\cite[Thm.~8.9]{DSEquiv}, gives the differential ranks.

For the remaining assertions we use the invariant $h$. Specialized to the torus knots $T(3,n)$, the recursion of
Daemi and Scaduto~\cite[Prop.~6.11]{DScaduto} reads
\[
h(T(3,n+3))+\tfrac12\sigma(T(3,n+3))=h(T(3,n))+\tfrac12\sigma(T(3,n))-2.
\]
With $\sigma=-2N$ and the values of $N$ above, $N(3,n+3)-N(3,n)$ is $3$ for $n\equiv1,2$ and $1$ for
$n\equiv4,5\pmod6$. So $h$ increases by one from $n\equiv1,2$ to $n+3\equiv4,5$, and decreases by one from
$n\equiv4,5$ to $n+3\equiv1,2\pmod6$. Starting from the unknot $T(3,1)$, whose S-complex is $\Z_{(0)}$ with zero differential, so that $h=0$, and the
trefoil $T(3,2)$, for which $h=0$~\cite[Thm.~1.14(i)]{DSEquiv}, this gives $h(T(3,n))=0$ for $n\equiv1,2$
and $h(T(3,n))=1$ for $n\equiv4,5\pmod6$. By~\cite[Prop.~4.15]{DSEquiv}, $h>0$ holds iff there is
$\alpha\in C_*$ with $d\alpha=0$ and $\delta_1(\alpha)\neq0$; since $d=0$, this gives $\delta_1\neq0$.
\end{proof}

\begin{proof}[Proof of Theorem~\ref{thm:C}]
Part~(i) is Proposition~\ref{prop:split}; part~(ii) is Proposition~\ref{prop:even}.
\end{proof}

\begin{proof}[Proof of Corollary~\ref{cor:torus}]
For a torus knot the traceless character variety consists of $|\sigma(K)|/2$ irreducible points and one
abelian point, all nondegenerate~\cite[\S9.5]{DSEquiv}, so $\dim\widetilde C(K)=1+|\sigma(K)|$. Its
homology has the rank of $\Inat(K)$~\cite[Thm.~8.9]{DSEquiv}, and
$\operatorname{rank}\Inat(K)=\dim KHI(K)=\|\Delta_K\|_1$ by~\cite[Prop.~1.4]{KM2} and~\cite[Cor.~1.8]{LiYe}
(Li and Ye work over $\C$, which does not change dimensions). The identity $\dim H=\dim
C-2\operatorname{rank}\td$ gives the formula.
\end{proof}

\begin{table}[ht]
\centering
\begin{tabular}{c|ccccccccccc}
$n$ & $2$ & $4$ & $5$ & $7$ & $8$ & $10$ & $11$ & $13$ & $14$ & $16$ & $17$\\
$n\bmod 6$ & $2$ & $4$ & $5$ & $1$ & $2$ & $4$ & $5$ & $1$ & $2$ & $4$ & $5$\\ \hline
$N(3,n)$ & $1$ & $3$ & $4$ & $4$ & $5$ & $7$ & $8$ & $8$ & $9$ & $11$ & $12$\\
$\dim\widetilde C=2N+1$ & $3$ & $7$ & $9$ & $9$ & $11$ & $15$ & $17$ & $17$ & $19$ & $23$ & $25$\\
$\operatorname{rank}\Inat=\|\Delta\|_1$ & $3$ & $\mathbf 5$ & $\mathbf 7$ & $9$ & $11$ & $\mathbf{13}$ &
$\mathbf{15}$ & $17$ & $19$ & $\mathbf{21}$ & $\mathbf{23}$\\
$\operatorname{rank}\td$ & $0$ & $1$ & $1$ & $0$ & $0$ & $1$ & $1$ & $0$ & $0$ & $1$ & $1$\\
\end{tabular}
\medskip
\caption{For $T(3,n)$: the traceless character count $N$, the rank $2N+1$ of the S-complex, the homology
rank $\operatorname{rank}\Inat=\|\Delta_{T(3,n)}\|_1$, and the total rank of the S-complex differential
(Proposition~\ref{prop:even}). The two ranks differ by $2$ (bold) exactly for $n\equiv4,5\pmod6$. The
smallest case with nonzero differential is $8_{19}=T(3,4)$, and the smallest odd case is
$T(3,5)=P(-2,3,5)=10_{124}$.}
\label{tab:3n}
\end{table}

For example, $T(3,8)$ has $N=5$ and chain and homology rank $11$, so its S-complex differential vanishes,
as~\cite[Prop.~9.37]{DSEquiv} shows for the family $(3,6k+2)$ with $k=1$; its homology rank exceeds
$\det=3$. For $T(3,5)=P(-2,3,5)$ the homology has rank $7$ rather than the chain rank $9$
\cite[\S9.3]{DSEquiv}.

\section{A nonzero differential: \texorpdfstring{$8_{19}=T(3,4)$}{8\_19 = T(3,4)}}\label{sec:819}

We compare the instanton complex of the smallest case with nonzero S-complex differential,
$8_{19}=T(3,4)$, with the pillowcase complex computed by Hedden--Herald--Kirk.

\subsection{The instanton complex}
By Theorem~\ref{thm:B} and Proposition~\ref{prop:count}, $T(3,4)$ has three irreducible traceless
characters, two non-dihedral (with $\cos\psi=\mp1/\sqrt3$) and one dihedral (with $\cos\psi=0$),
together with one abelian character. On the double cover $\Sig(2,3,4)$, whose first homology is $\Z/3$,
the two non-dihedral characters lie over the unique irreducible flat connection, and the dihedral
character lies over the nontrivial reducible (abelian) flat connection~\cite[\S9.1]{DSEquiv}. Here $\sigma=-6$,
and~\cite[Cor.~3]{DScaduto} gives $C_*=\Z^2_{(1)}\oplus\Z_{(3)}$, with the two non-dihedral characters in
degree $1$ and the dihedral one in degree $3$~\cite[\S9.4]{DSEquiv}. Thus $\widetilde C(8_{19})$ has rank
vector $(2,2,2,1)$ in Daemi--Scaduto's normalization, and $(2,1,2,2)$ after the shift by
$\sigma\equiv2\pmod4$ that matches the grading of $\Inat$~\cite[Thm.~8.9]{DSEquiv}, as
in~\cite[Ex.~7.9]{PS}. The graded homology $\Inat(8_{19})$ has rank vector $(2,1,1,1)$~\cite[\S12.2.3]{HHK1},
\cite[Ex.~7.8]{PS}. Comparing ranks grading by grading, a differential of degree $-1$ must have exactly
one nonzero graded component, of rank one, from $\Inat$-grading $3$ to $\Inat$-grading $2$, that is, from
$\gr_{DS}=1$ to $\gr_{DS}=0$. Daemi and Scaduto identify it: $\delta_1$ is nonzero on $C_1$, mapping the
two characters lying over the irreducible connection to the reducible generator~\cite[\S9.4]{DSEquiv},
and $h(T(3,4))=1$~\cite[Thm.~1.14(iii)]{DSEquiv}. Poudel and Saveliev state the same map in a
cohomological convention, as a nontrivial map of degree $+1$~\cite[Ex.~7.9]{PS}.

\subsection{The pillowcase complex}
Hedden--Herald--Kirk compute the corresponding pillowcase complex~\cite[\S11.6]{HHK2}. They use a
decomposition of $T(3,4)$ along a Conway sphere bounding a ball in which the knot is a trivial tangle,
with the parameters $(p,q,r,s)=(3,4,3,-2)$ of~\cite[\S10]{HHK2}, where the integers $r,s$ with $pr+qs=1$ fix the
decomposition; $L_0$ is the earring figure-eight of the trivial
tangle, and $L_1$ is the perturbed character variety of the other
tangle. The unperturbed variety is homeomorphic to the letter $\varphi$, a circle with an arc through
it~\cite[\S11.6]{HHK2},~\cite[Prop.~11.4]{HHK1}, and under perturbation it resolves into an arc $R_0$ and
a circle $R_1$ that winds twice around the pillowcase in the $\theta$-direction, monotonically. Write
$\gr_{\pill}$ for the $\Z/4$ grading of the pillowcase complex~\cite[Def.~3.7]{HHK2}: on each component of $L_1$ it
is a relative grading computed from Maslov indices of paths in $L_0$ and $L_1$~\cite[Lemma~3.10]{HHK2}, a bigon
lowers it by one~\cite[Prop.~3.8]{HHK2}, and on the component through $r_+$ it is made absolute by placing $r_+$
in grading $\sigma\bmod4=2$~\cite[(19), \S11.6]{HHK2}. The seven generators are the three points $r_+$, $x_1^+$, $x_1^-$ where $R_0$ meets
$L_0$, with $r_+$ the perturbation of the abelian character, near the corner $(0,0)$, and four generators
on $R_1$; here $\gr_{\pill}(x_1^+)=0$ and $\gr_{\pill}(x_1^-)=3$, and $R_1$ contributes one generator in each
grading. There is exactly one bigon, from $x_1^-$ to $r_+$, so $\partial x_1^-=r_+$, and the homology is
\[
(1,0,0,0)_{R_0}\oplus(1,1,1,1)_{R_1}=(2,1,1,1),
\]
in agreement with $\Inat(8_{19})$. The bigon has one corner at $r_+$, near $(0,0)$, and the other at
$x_1^-$, near the diagonal~\cite[Fig.~21]{HHK2}.

\subsection{Comparison}
The two computations exhibit a rank-one map of the same degree, from grading $3$ to grading $2$, ending at
the generator of the abelian character. They concern different complexes: no chain-level identification
of the pillowcase complex with $\widetilde C$ is known in general, and the conjectures recalled in
Section~\ref{sec:intro-af} concern homology. A chain-level conjecture over $\mathbb F_2$ for torus knots is
formulated in~\cite[Conj.~1.5]{WuebbenIV}, and for $T(3,n)$ with $n\ge4$ and perturbations in the cone
of~\cite[Def.~5.5]{WuebbenV} it is proved in~\cite[Thm.~1.3]{WuebbenV}. In the S-complex structure on the
pillowcase complex used there~\cite[\S1.3]{WuebbenIV}, with $r_+$ as the reducible generator and $x_i^-$, $x_i^+$
as the two copies of an irreducible generator, the differential is a single component of type $\delta_1$, of
the same rank as $\td$; for $n\equiv4,5\pmod6$ it is given by one bigon, ending at $r_+$. The pillowcase differential also depends on the perturbation: for a
larger perturbation parameter the two generators joined by the bigon cancel and the perturbed complex is
perfect~\cite[\S11.10, Prop.~11.1]{HHK2}. By contrast, $\delta_1\neq0$ in every S-complex representing
$\widetilde C(T(3,4))$, because $h(T(3,4))=1$~\cite[Prop.~4.15]{DSEquiv}. The same kind of bigon, from
$x_1^-$ to $r_+$, occurs in Hedden--Herald--Kirk's computation for $T(3,5)$~\cite[\S11.7]{HHK2}, whose
double branched cover has no nontrivial reducible flat connection, so the bigon is not tied to the nontrivial reducible of
$\Sig(2,3,4)$; for two-bridge knots $L_1$ is straight and there is no bigon
(Section~\ref{sec:2bridge-pillow}). For $8_{19}$ the chain rank $7$ exceeds $\det=3$ by the four generators
of the two non-dihedral characters, and the differential reduces it to $\operatorname{rank}\Inat(8_{19})=5$.

\section{Open questions}\label{sec:open}

On the instanton side, the invariants considered here are determined for every $T(3,n)$: the rank of the
S-complex differential is zero or one according to $n\bmod6$, and in the rank-one cases its component $\delta_1$ is
nonzero (Theorem~\ref{thm:C}). On the pillowcase side, for the decomposition of~\cite[\S10]{HHK2} with the
integers $(r,s)$ used there for $T(3,4)$, $T(3,5)$ and $T(3,7)$, and for sufficiently small perturbations in
the closed cone of~\cite[Def.~5.5]{WuebbenV}, the pillowcase Floer homology of every $T(3,n)$ with $n\ge4$ has
the $\Z/4$-graded dimensions of $\Inat$, so the conclusion of~\cite[Conj.~6.5]{HHK2} holds for this
family~\cite[Cor.~1.2]{WuebbenV}; after its grading is shifted by $\sigma$ and its closed components are
suitably graded, the pillowcase complex is S-chain homotopy equivalent to
$\widetilde C(T(3,n))\otimes\mathbb F_2$~\cite[Thm.~1.3]{WuebbenV}. Hedden--Herald--Kirk computed
pillowcase Floer homology with zero bounding cochains, of the correct rank, for $T(3,4)$, $T(3,5)$ and
$T(3,7)$, and for the $(4,5)$-torus knot~\cite[\S\S11.2, 11.6--11.8]{HHK2}; the evidence that other
decompositions need nonzero cochains is recalled in Remark~\ref{rmk:earring}. The companion paper~\cite{WuebbenII} determines $\Inat(P(-2,3,q);\Z)$ for the pretzel knots $P(-2,3,q)$, a family
containing $T(3,4)=P(-2,3,3)$ and $T(3,5)=P(-2,3,5)$, and proves~\cite[Conj.~6.5]{HHK2} for a decomposition of each
$P(-2,3,q)$ with $q\ge7$. For Smith's decomposition of the same knots, Maurer--Cartan deformations of explicit
immersed curves in an algebraic model of the wrapped Fukaya category of the pillowcase with its four corners
removed are studied in~\cite{WuebbenIIb}.

We list three open questions; Conjecture~D of~\cite{CHKK} is recalled in Section~\ref{sec:intro-af}.
\begin{enumerate}
\item Construct the bounding cochains of~\cite[Conj.~D]{CHKK} for decompositions in which zero cochains give
the wrong rank, such as those recalled in Remark~\ref{rmk:earring}, and prove that they restore the rank of
$\Inat$. For the second decomposition of $T(3,4)$ in that remark, a cochain with this property is found
numerically in~\cite[\S7]{WuebbenV}.
\item Extend the chain-level identification over $\mathbb F_2$ from the knots $T(3,n)$ to all torus knots, as
conjectured in~\cite[Conj.~1.5]{WuebbenIV}.
\item Identify the Maurer--Cartan deformations of~\cite{WuebbenIIb} with the bounding cochains of~\cite{CHKK}.
\end{enumerate}

\bigskip
\noindent{\sc Bernd J. Wuebben, New York, NY,} \texttt{wuebben@gmail.com}

\end{document}